\documentclass[12pt]{article}
\usepackage{amsmath,amsthm,amsfonts,amssymb,verbatim}

\title{Notes on Chain Recurrence and Lyapunonv Functions}
\author{John Franks}

\newtheorem{theorem}{Theorem}[section]

\newtheorem{defn}[theorem]{Definition}

\newtheorem{exer}[theorem]{Exercise}
\newtheorem{thm}[theorem]{Theorem}
\newtheorem{cor}[theorem]{Corollary} 
\newtheorem{lem}[theorem]{Lemma}
\newtheorem{prop}[theorem]{Proposition}

\newcommand{\eps}{{\varepsilon}}
\newcommand{\R}{{\cal R}}
\newcommand{\B}{{\cal B}}

\newcommand{\rr}{{\mathbb R}}

\newcommand{\N}{\mathcal N}

\begin{document}
\maketitle
\begin{abstract}
This short expository note provides an introduction to the concept of chain recurrence in topological
dynamics and a proof of the existence complete Lyapunov functions for homeomorphisms of
compact metric spaces due to Charles Conley \cite{C}.  
I have used it as supplementary material in  introductory dynamics courses.
\end{abstract}

\section
{Epsilon Chains}

We briefly review the definition of $\eps$-chains and chain
recurrence developed by Charles Conley in \cite{C}.  In the following
$f: X \to X$ will denote a homeomorphism of a compact metric space
$X$.

\begin{defn}  An {\em $\eps$-chain\/} from $x$ to
$y$ for $f$  is a sequence of points in $X$, 
$x = x_0,x_1,\dots,x_n =y$,  with $n \ge 1$, such that 
$$d(f(x_i),x_{i+1}) <\eps \qquad\text{for } 0\le i\le n-1.$$ A point
$x\in X$ is called {\em chain recurrent} if for every $\eps>0$ there
is an $\eps$-chain from $x$ to itself.  The set $\R(f)$ of chain
recurrent points is called the {\it chain recurrent set\/} of $f$.
\end{defn}

\begin{exer}
Let $f: X \to X$ be a homeomorphism of a compact metric space.
\begin{enumerate}
\item The set $\R(f)$ is closed (hence compact) and invariant under $f$.  

\item If $x_0,x_1,\dots,x_n$ is an $\eps$-chain from $x$ to $y$
and $y_0,y_1,\dots,y_m$ is an $\eps$-chain from $y$ to $z$, then
$x_0,x_1,\dots,x_n = y_0, y_1,y_2,\dots,y_m$ is an $\eps$-chain from $x$ to $z$.

\item If for every $\eps > 0$ there is an $\eps$-chain from $x$ to
$y$ for $f$ then for every $\eps > 0$ there is a $\eps$-chain from
$y$ to $x$ for $f^{-1}.$

\end{enumerate}

\end{exer}

Recall that a point $x$ is called {\em recurrent} for $f:X \to X$
if $x$ is a limit point of the sequence $x, f(x), f^2(x), \dots
f^n(x), \dots.$  Clearly any recurrent point is also chain recurrent.
The converse is not true.  

Recall that if $\mu$ is a finite 
Borel measure  on $X$ and $f: X \to X$  is a,
not necessarily invertible, function then
we say $\mu$ is $f-$invariant provided
$\mu(E) = \mu(f^{-1}(E))$ for every measurable
subset $E \subset X.$

If there is a finite $f$-invariant measure
on $X$ then almost every point
of $X$ (in the measure sense) is recurrent.  

\begin{thm}[Poincar\'e Recurrence Theorem]\label{thm: PR}
Suppose $\mu$ is a finite 
Borel measure on $X$ and $f: X \to X$  is a 
measure preserving transformation.
If $E \subset X$ is measurable and $\N$ is the subset of 
$E$ given by
\[
\N = \{ x \in E \ | \ f^k(x) \in E \text{ for at most
finitely many } k \ge 1\},
\]
then $\N$ is measurable and $\mu(\N) = 0.$
\end{thm}
\begin{proof}
Define
\[
E_N = \bigcup_{n=N}^\infty f^{-n}(E) \text{ and }
F = \bigcap_{n=0}^\infty E_n.
\]
Then $F$ is the set of points whose forward orbit hits $E$
infinitely often so $\N= E\setminus F$ and $\N$ is measurable.
Since $E_{n +1} = f^{-1}(E_n)$
we have $\mu(E_{n +1}) = \mu(E_n)$ for all $n \ge 0.$
Since $E_0 \supset E_1 \supset E_2 \dots$ we have
\[
\mu(F) = \mu( \bigcap_{n=0}^\infty E_n)
= \lim_{n \to\infty}\mu(E_n) = \mu(E_0).
\]
Hence $\mu(E_0 \setminus F) = \mu(E_0) - \mu( F) =0.$ 
Since $E \setminus F \subset E_0 \setminus F$
we conclude
$\mu(\N) = \mu(E \setminus F) = 0.$
\end{proof}

\begin{cor}
Suppose $\mu$ is a probability measure on $X$.
If $\mu$ is $f$-invariant then the set $\N$ of 
points which are not recurrent has measure $\mu(\N) = 0.$
\end{cor}

\begin{proof}
Let $\N_n$ denote the set of points $x \in X$ such that
$d(x, f^k(x)) > 1/n$ for all $k>0.$  We wish first to show
$\mu(\N_n) = 0$ for all $n> 0.$  

To do this suppose $B$ is an open ball in the metric space $X$ of
radius $1/2n$ so the distance between any two points of $B$ is less
than $1/n.$ We conclude from Theorem~(\ref{thm: PR}) that $\mu(B \cap
\N_n) = 0.$ But since $X$ is compact it can be covered by finitely
many balls $B$ of radius $1/2n$ so we conclude $\mu(\N_n) = 0.$
Since 
\[
\N = \bigcup_{n=1}^\infty \N_n
\]
we conclude $\mu(\N) = 0.$
\end{proof}

We have the following immediate corollary.

\begin{cor}
Suppose $f:X \to X$ preserves a finite Borel measure $\mu$
and $\mu(U) > 0$ for every non-empty open set $U$. Then there
are recurrent points in every such $U$.  I.e.
the recurrent points are dense in $X$.
\end{cor}

It is easy to see
that in this case the chain recurrent set $\R(f)$ is all
of $X$.  Also, as we now show, in this circumstance if $X$ is
connected, then for any points
$x,y \in X$ there is an $\eps$-chain  from $x$ to $y$.

\begin{prop} Suppose $\mu$ is an
$f$-invariant measure on $X$ satisfying $\mu(X) =1$ and $\mu(U) > 0$
for every non-empty open set $U \subset X$ and suppose that $X$ is
connected.  Then for any $x,y \in X$ and any $\eps > 0,$ there is
an $\eps$-chain from $x$ to $y$.
\end{prop}

\begin{proof}  Fix a value of $\eps >0.$  
We construct an equivalence relation $\sim_\eps$ on the space $X$ as
follows.  Let $x \sim_\eps y$ provided there is an $\eps$-chain from
$x$ to $y$ and one from $y$ to $x.$ This clearly defines a symmetric
and transitive relation.  It is reflexive as well, however.  To see
this, let $U$ be a neighborhood of $x$ such that
$U$ and $f(U)$ have diameter less than $\eps.$ Clearly if $f^n(U)
\cap U \ne \emptyset$ for some $n>0$ then there is an $\eps$-chain
from $x$ to $x$.  In fact, if $x_0,f^n(x_0)\in U,$ we can define $x_1 =
x,\ x_i = f^i(x_0), \ 1 < i < n$ and the only ``jumps'' needed are
from $f(x)$ to $x_2 = f(x_0)$ and from $f(x_{n-1}) = f^n(x_0)$ to $x_n
= x.$ 

But for any open $U$ it must be the case that
$f^n(U) \cap U \ne \emptyset$ for some $n >0$
since otherwise the sets $f^i(U)$ are pairwise disjoint and all have
the same positive measure which would mean the measure of $X$ is
infinite.  Hence the relation $\sim_\eps$ is reflexive and thus an
equivalence relation.

From the definition of $\eps$-chain it is immediate
that the equivalence classes are open sets in $X$.  Since the
equivalence classes form a partition of $X$ into pairwise disjoint
open sets and $X$ is connected, there must be a single equivalence
class.  Thus for any $x,y \in X$ there is an $\eps$-chain for $X$
from $x$ to $y$.  Since $\eps$ was arbitrary the result follows.
\end{proof}

\section{The ``Fundamental Theorem of Dynamical Systems''}

In this section we briefly review the elementary theory of
attractor-repeller pairs and complete Lyapunov functions developed by
Charles Conley in \cite{C}.  We give Conley's proof of the the
existence of complete Lyapunov functions, (which is sometimes called
the ``Fundamental Theorem of Dynamical Systems'')

If $A\subset X$ is a compact subset and there is an open neighborhood $U$ of 
$A$ such that $f(cl(U))\subset U$ and $\bigcap_{n\ge0}f^n(cl(U))=A$, then $A$ 
is called an {\it attractor\/} and $U$ is an isolating neighborhood. It is 
easy to see that if $V=X \setminus cl(U)$ and $A^*=\bigcap_{n\ge0}f^{-n} (cl(V))$, then 
$A^*$ is an attractor for $f^{-1}$ with isolating neighborhood $V$. The set 
$A^*$ is called the {\it repeller\/} dual to $A$. It is clear that $A^*$ is 
independent of the choice of isolating neighborhood $U$ for $A$. Obviously 
$f(A)=A$ and $f(A^*)=A^*$.

\begin{lem}\label{2.1} The set of attractors for $f$ is countable. 
\end{lem}

\begin{proof} Choose a countable basis $\B=\{V_n\}^\infty_{n=1}$ for the topology 
of $X$. If $A$ is an attractor with open isolating neighborhood $U$, then $U$ 
is a union of sets in $\B$. Hence, since $A$ is compact, there are 
$V_{i_1},\dots,V_{i_k}$ such that $A\subset V_{i_1}\cup\cdots\cup V_{i_k}\subset 
U$. Clearly $A=\bigcap_{n\ge0} f^n(U)=\bigcap_{n\ge0} f^n(V_{i_1}\cup 
\cdots\cup V_{i_k})$. Consequently there are at most as many attractors as 
finite subsets of $\B$, i.e., the set of attractors is countable.
\end{proof}

\begin{lem}\label{lem: dual}
If $\{A_n\}^\infty_{n=1}$ are the attractors of $f$ and $\{A^*_n\}$
their dual repellers, then the chain recurrent set
$\R(f)=\bigcap^\infty_{n=1} (A_n\cup A^*_n)$.
\end{lem}

\begin{proof} We first show $\R(f) \subset\cap(A_n \cup A^*_n)$.
This is equivalent to showing that if $x\notin A\cup A^*$ for some
attractor $A$, then $x\notin \R(f)$. If $U$ is an open isolating
neighborhood of $A$ and $x\notin A\cup A^*$, then $x\in f^{-n}(U)$ for
some $n$. Let $m$ be the smallest such $n$.  Replacing $U$ with
$f^{-m}(U)$ we can assume $x\in U \setminus f(U)$.  Now choose
$\eps_0>0$ so that any $\eps_0$-chain $x=x_1,x_2,x_3$ must
have $x_3\in f^2(U)$. If $\eps_1=d(X \setminus f(U), cl(f^2 (U)))$
and $\eps=\frac12\,\min\{\eps_0,\eps_1\}$, then no
$\eps$-chain can start and end at $x$, since no $\eps$-chain
from a point of $f^2(U)$ can reach a point of $X \setminus f(U)$. Thus
$x\notin \R(f)$. We have shown $\R(f)\subset\cap (A_n\cup A^*_n)$.

We next show the reverse inclusion. Suppose $x\in\bigcap^\infty_{n=1} (A_n\cup 
A^*_n)$. If $x$ is not in $\R(f)$, there is an $\eps_0>0$ such that no 
$\eps_0$-chain from $x$ to itself exists. Let $\Omega(x,\eps)$ denote the 
set of $y\in X$ such that there is an $\eps$-chain from $x$ to $y$. By 
definition, the set $V=\Omega(x,\eps_0)$ is open. Moreover, $f(cl(V))\subset 
V$, because if $z\in cl(V)$, there is $z_0\in V$ such that 
$d(f(z),f(z_0))<\eps_0$ and consequently an $\eps_0$-chain from $x$ to 
$z_0$, gives an $\eps_0$-chain $x=x_0,x_1,\dots,x_k,z_0,f(z)$ from $x$ to 
$f(z)$. Hence $A=\bigcap_{n\ge0} f^n(cl(V))$ is an attractor with isolating 
neighborhood $V$. By assumption either $x\in A$ or $x\in A^*$. Since there is 
no $\eps_0$-chain from $x$ to $x$, $x\notin A$. On the other hand, if 
$\omega(x)$ denotes the limit points of $\{f^n(x)\bigm| n\ge0\}$, then clearly 
$\omega(x)\subset V$, but this is not possible if $x\in A^*$ since $A^*$ is closed 
and $x\in A^*$ would imply $\omega(x)\subset A^*$. Thus we have contradicted 
the assumption that $x\notin\R$.
\end{proof}

\begin{exer}
Let $f= id : X \to X$ be the identity homeomorphism of a compact metric space.
Find all attractors of $f$ and their dual repellers.
\end{exer}

If we define a relation $\sim$ on $\R$ by $x\sim y$ if for every $\eps>0$ 
there is an $\eps$-chain from $x$ to $y$ and another from $y$ to $x$, then 
it is clear that $\sim$ is an equivalence relation. 

\begin{defn} The equivalence classes in $\R(f)$ for the 
equivalence relation $\sim$ above are called the {\em chain transitive 
components\/} of $\R(f)$.
\end{defn}

\begin{prop}\label{2.4}
If $x,y\in\R(f)$, then $x$ and $y$ are in the same 
chain transitive component if and only if there is no attractor $A$ with $x\in 
A$, $y\in A^*$ or with $y\in A$, $x\in A^*$. 
\end{prop}

\begin{proof} 
Suppose first that $x$ and $y$ are in the same chain transitive
component, i.e., $x\sim y$, and $x\in A$. If $U$ is an open isolating
neighborhood for $A$, let $\eps=\text{ dist}(X \setminus U,
cl(f(U)))$. There can be no $\eps/2$-chain from a point in $f(U)$
to a point in $X \setminus U$, hence none from a point in $A$ to a
point in $A^*$. By Lemma~(\ref{lem: dual}) we know $y\in A\cup A^*$, but $x\sim
y$ implies $y\notin A^*$, so $y\in A$. This proves one direction of
our result.

To show the converse, suppose that for every attractor $A$, $x\in A$ if and only if $y \in 
A$ (and hence $x\in A^*$ if and only if $y\in A^*$). Given $\eps>0$ let 
$V=\Omega(x,\eps)=$ the set of all points $z$ in $X$ for which there is an 
$\eps$-chain from $x$ to $z$. Since $x$ is chain recurrent $x\in V$.Also as 
in the proof of Lemma~(\ref{lem: dual}) $V$ is an isolating neighborhood for an attractor $A_0$. 
Since $x\in A_0\cup A^*_0$ and $x\in V$ we have $x\in A_0$. Thus $y\in 
A_0\subset V$ so there is an $\eps$-chain from $x$ to $y$. A similar 
argument shows there is an $\eps$-chain from $y$ to $x$ so $x\sim y$. 
\end{proof}

We are now prepared to present Conley's proof of the existence of a complete 
Lyapunov function.

\begin{defn} A complete Lyapunov function for $f: X\to X$ is a 
continuous function $g: X \to \rr$ satisfying:
\begin{enumerate}
\item If $x\notin \R(f)$, then $g(f(x))<g(x)$

\item If $x,y\in\R(f)$, then $g(x)=g(y)$ if and only if $x\sim y$ (i.e., $x$ and 
$y$ are in the same chain transitive component.

\item $g(\R(f))$ is a compact nowhere dense subset of $\rr$.
\end{enumerate}
By analogy with the smooth setting, elements of $g(\R(f))$ are called {\em
critical values} of $g$. 
\end{defn}

\begin{lem}\label{2.6}
There is a continuous function $g: X\to[0,1]$ such that 
$g^{-1}(0)=A$, $g^{-1}(1)=A^*$ and $g$ is strictly decreasing on orbits of 
points in $X \setminus (A\cup A^*)$. 
\end{lem}

\begin{proof} Define $g_0: X\to[0,1]$ by
$$g_0(x)=\frac{d(x,A)}{d(x,A)+d(x,A^*)}.$$
Let $g_1(x)=\sup\{g_0 (f^n(x))\bigm| n\ge0\}$. Then $g_1: X\to[0,1]$ and 
$g_1(f(x))\le g_1(x)$ for all $x$. We must show $g_1$ is continuous. If $\lim 
x_i=x\in A$, then clearly $\lim g_1(x_i)=0$ so $g_1$ is continuous at points 
of $A$ and the same argument shows it is continuous at points of $A^*$. If $U$ 
is an open isolating neighborhood as above, let $N=cl(U) \setminus f(U)$. Let $x\in N$ 
and $r=\inf\{g_0(x)\bigm| x\in N\}$. Since $f^n(N)\subset f^n(cl(U))$ and 
$\bigcap_{n\ge0}f^n (cl(U))=A$, it follows that there is $n_0>0$ such that 
$g_0(f^n(N))\subset [0,r/2]$ whenever $n>n_0$. Hence for $x\in N$,
$$g_1(x)=\max\{g_0(f^n(x))\bigm| 0\le n\le n_0\}$$
so $g_1$ is continuous on $N$. Since $\bigcup^\infty_{n=-\infty} f^n(N)=X \setminus 
(A\cup A^*)$, $g_1$ is continuous. Finally, letting $$g(x)=\sum^\infty_{n=0} 
\frac{g_1(f^n(x))}{2^{n+1}}$$ we obtain a continuous function $g: X\to [0,1]$ 
such that $g^{-1}(0)=A$, $g^{-1}(1)=A^*$. Also 
$$g(f(x))-g(x)=\sum^\infty_{n=0} \frac{g_1(f^{n+1}(x)) - g_1(f^n(x))}{2^{n+1}}
$$
which is negative if $x\notin A\cup A^*$, since $g_1(f(y))\le g_1(y)$ for all 
$y$ and $g_1$ is not constant on the orbit of $x$.
\end{proof}

The following theorem is essentially a result of [C]. We have changed the 
setting from flows to homeomorphisms.

\begin{thm} [Fundamental Theorem of Dynamical Systems]
 If $f: X\to X$ is a homeomorphism of a compact metric 
space, then there is a complete Lyapunov function $g: X\to \rr$ for $f$.
\end{thm}

\begin{proof} 
By Lemma~(\ref{2.1}) there are only countably many attractors
$\{A_n\}$ for $f$. By Lemma~(\ref{2.6}) we can find $g_n: X\to \rr$
with $g^{-1}_n(0)=A_n$, $g^{-1}_n(1)= A^*_n$ and $g_n$ strictly
decreasing on $X \setminus (A_n\cup A^*_n)$. Define $g: X\to \rr$
by $$g(x)=\sum^\infty_{n=1} \frac{2g_n(x)}{3^n}.$$ The series
converges uniformly so $g(x)$ is continuous. Clearly if
$x\notin\R(f)$, then there is an $A_i$ with $x\notin (A_i\cup A^*_i)$
so $g(f(x))<g(x)$.

Also, if $x\in\R(f)$, then $x\in(A_n\cup A^*_n)$ for every $n$, so $g_n(x)=0$ 
or 1 for all $n$. It follows that the ternary expansion of $g(x)$ can be 
written with only the digits 0 and 2, and hence $g(x)\in C$, the Cantor 
middle third set. Thus $g(\R(f))\subset C$ so $g(\R(f))$ is compact and 
nowhere dense. This proves (3) of the definition.

Finally, if $x,y\in\R(f)$ then $g(x)=g(y)$ if and only if , 
$g_n(x)=g_n(y)$ for all $n$. This is true since $2g_n(x)$ 
is the $n^{\text{th}}$ digit of the ternary 
expansion of $g(x)$  so $g(x)=g(y)$ implies $g_n(x)=g_n(y)$ for all $n$.
But $g_n(x)=g_n(y)$ for all $n$ if and only if  there is
no $n$ with $x\in A_n$, $y\in A^*_n$ or with $x\in A^*_n$, $y\in A_n$. Thus by 
Proposition~(\ref{2.4}), $g(x)=g(y)$ if and only if $x$ and $y$ are in the same chain transitive 
component.
\end{proof}

\end{document}